\documentclass[10pt]{amsart}

\usepackage{tasks}
\usepackage{hyperref}
\usepackage{blindtext}
\usepackage{setspace}
\usepackage{marginnote}
\usepackage{color}
\usepackage{amsrefs}
\usepackage{graphicx}  
\usepackage[top=3.5cm, bottom=2.5cm, outer=2.5cm, inner=2.5cm]{geometry}
\usepackage{hyperref}
\hypersetup{linkcolor=red,
citecolor=black
}
\usepackage{amsmath}

\newtheorem{thm}{Theorem}[section]
\newtheorem{lem}[thm]{Lemma}
\newtheorem{prop}[thm]{Proposition}
\newtheorem{cor}[thm]{Corollary}

\theoremstyle{definition}

\theoremstyle{remark}
\newtheorem{re}[thm]{Remark}

\numberwithin{equation}{section}

\usepackage[utf8]{inputenc}
\usepackage[english]{babel}
\usepackage{paralist}
\usepackage{amsthm}
\usepackage{mathtools}
 
\theoremstyle{definition}
\newtheorem{definition}{Definition}[section]
 
\theoremstyle{remark}

 \usepackage{bm,amsmath}

\newcommand{\RN}[1]{%
  \textup{\uppercase\expandafter{\romannumeral#1}}%
}

\newcommand{\Rey}{\mathcal{R}e }

\newcommand{\grad}{\nabla}

\allowdisplaybreaks[4]
\begin{document}

\title{On the long time behavior of  time relaxation model of fluids}


\author{Ali Pakzad}
\address{Department of Mathematics, University of California,  Riverside, CA 92521, USA}
\email{alip@ucr.edu}

\maketitle
\setcounter{tocdepth}{1}

\begin{abstract}

The time relaxation model, which is  family of high accuracy turbulence models, has proven to be effective in regularization of Navier–Stokes Equations.  The model belongs to the class of Large Eddy Simulation  models, and  is derived by  adding  a  linear  time  regularization term $\chi u^{\star}$ to the Navier–Stokes Equations.  The  time relaxation operator  truncates small solution scales by injecting an extra dissipation to a simulation,  without altering appreciably the solution's large scales. Herein to evaluate the effect of the time  regularization term on a simulation,  the rate of energy dissipation of the model in body-force-driven  turbulence is studied.  Our result, which agrees with Kolmogorov’s conventional turbulence theory, is also consistent with the rate proven for the NSE. Moreover, employing the model  requires a choice of the coefficient $\chi$. It is known that the model's simulation is sensitive to the parameter.  The analysis motivates a range of possible values for the coefficient $\chi$ in $3d$ turbulent flows away from walls.

\end{abstract}

\section{Introduction}
A distinctive feature of turbulent flows is the emergence of complicated chaotic structures involving a wide range of length scales.  Based on K-41 theory for a  $3 D$ turbulence, \cite{F95} and \cite{P00}, capturing all these scales typically requires $\mathcal{O} (\Rey^{\frac{9}{4}})$ mesh points in space per time step for a direct numerical simulation of the  Navier–Stokes equations ($\chi = 0$ in (\ref{TRM0})).  Such calculations  are infeasible for practical problems at even modest Reynolds number. On the other hand, using a coarse discretization $\simeq \mathcal{O}(\delta)$ can  lead to the non-physical temporal growth of the fluctuations   due to  neglecting the dissipation that occurs at very small scales (smaller than the typical coarse mesh).  To \textit{relax} these difficult discretization requirements, several numerical regularization techniques have been developed for simulations. Time relaxation models (TRM), which were  introduced by Stolz, Adam and Kleiser in \cite{SAK01-1} and \cite{SAK01-2},  are a novel class of regularization of the Navier–Stokes equations (NSE).  The model is accomplished by adding a time relaxation operator, as a numerical regularization, to the momentum equation of the NSE,
\begin{equation}\label{TRM0}
u_t+  \nabla \cdot (u \otimes u) -  \nu \Delta u + \chi u^{\star} +  \nabla p = f(x),
\end{equation}
where $u$ represents the fluid velocity, $p$ is the pressure, $\nu$ is kinematic viscosity, and $f$ accounts for external forcing. In (\ref{TRM0}) $u^{\star}$ is a generalized fluctuation  over length scales less than $\mathcal{O}(\delta)$,  and $ \chi > 0$ is  the scaling parameter which  has the units $[\mbox{time}]^{-1}$. Broadly speaking, $ \chi u^{\star}$ is intended  to strongly damp the non-physical unresolved fluctuations $< \mathcal{O}(\delta)$, without altering appreciably the solution's large scales $\geq\mathcal{O}(\delta)$. Numerical  experience with the model also  indicates a significant improvement over classical subgrid scale models   with a lower computational cost (e.g. \cite{ELN07}, \cite{SAK01-1} and  \cite{SAK01-2}).

On the other hand,  turbulence models seek to predict flow statistics (long time averages)   instead of  individual trajectories (\cite{H72}, \cite{MW06} and \cite{W10}). Indeed, much of the classical turbulence theories, such as the famous Kolmogorov’s conventional turbulence theory, are presented in the statistical forms (\cite{FMRT01} and \cite{F95}).  One quantity of great interest in applications and  importance in the study of the statistical properties  for turbulent flows (in the sense of J. Leray) is  the time averaged energy dissipation rate \cite{P00}. In this paper, we consider the Time Relaxation Model to calculate  statistics of the energy dissipation rate  of the large eddies in the  turbulent fluid in $3D$ in the absence of boundaries. 
 
 \subsection{Related Works}
 The energy dissipation rate is a fundamental statistic in experimental and theoretical studies of turbulence (\cite{F95} and \cite{P00}). Recently, there has been significant progress in deriving bounds on the time-averaged energy dissipation rate for turbulent flows  for incompressible homogeneous Newtonian fluids.  Upper bounds at every instant of
time yield estimates on the small length scales in the solutions (Wang \cite{W00}).  Kolmogrov first argued that at large Reynolds number, the energy dissipation rate per unit volume should be independent of the kinematic viscosity. Based on the concept of the energy cascade, and by a dimensional consideration, the energy dissipation rate per unit volume must take the form constant times $\frac{U^3}{L}$ (Frisch \cite{F95}), where $U$ and $L$ are global velocity and length scales. 

Doering and Constantin \cite{DC92} first established a rigorous upper bound for the time averaged energy dissipation rate for shear flows directly from the NSE.  Similar estimations have been proven by Marchiano \cite{M94}, Wang \cite{W97} and Kerswell \cite{K97} in more generality.  The result of Doering and Constantin has been also  generalized to other turbulence models in LES by Pakzad \cite{AP16} and Layton \cite{L02}.  
The effect of the mesh size on turbulence statistics was studied in \cite{AP18} for  discretized flow equations. 

 In non-equilibrium steady state the rate of energy dissipation must be balanced by the rate of work done by external forces to the system (Doering and Gibbon \cite{DG95}).   For  body-force-driven steady-state turbulence  Doering and  Foias \cite{DF02} delineated  bounds on the bulk rate of energy dissipation directly from the NSE, 
 $$\langle \varepsilon \rangle \leq (1+ \Rey^{-1}) \frac{U^3}{L}.$$
 Their result  has been extended to  other turbulence models and regularizations  in \cite{DLPRSZ18}, \cite{L16} and \cite{LRS10}.
 
On the other hand, a  model's performance depends on the choices for non-physical quantities like the relaxation coefficient.  It has been observed that the perfromance is  sensitive to the parameter  \cite{NPW15}.
In \cite{CL10} optimizing the error in discrete deconvolution suggests the  scaling  $\chi \simeq \delta^{-2}$.  After developing a similarity theory for the Time Relaxation model  following the $K-41$ theory of the Navier-Stokes equations, Layton and Neda \cite{LN07} proposed scaling $\chi \simeq \delta^{-\frac{2}{3}}$ by combining a mix of physical insight, mathematical analysis and dimensional analysis. 
 
 We begin in Section \ref{Section2}  where we briefly introduce basic notations and preliminaries,  and give a precise definition of the averaging operator and the higher-order approximate deconvolution that are used to define the generalized fluctuation $u^{\star}$. Section \ref{Section3} gives the analysis calculating the energy dissipation of the model,  and the major results are proven. We propose, based on an analysis
of the energy dissipation, a narrowing of the commonly accepted ranges of parameter $\chi$;  results are summarized below.   Section \ref{Section4}  collects conclusions and open problems.

\subsection{Summary of Results}
For body force driven turbulence, we prove  the following bounds on the  time-averaged energy dissipation rate  $\langle \varepsilon \rangle$ directly from the model,

$$\langle \varepsilon \rangle \leq  \bigg(2+ \Rey^{-1} + \frac{\chi \delta^2}{U L}\bigg) \, \frac{U^3}{L},$$
where $U, L$ are global velocity and length scales, respectively, and $\delta$ is the large eddy simulation filter radius.  In this  estimate  $\langle \varepsilon \rangle$ balances the energy input rate, $\frac{U^3}{L}$. This estimate is also consistent as $\Rey \rightarrow \infty ,  \delta\rightarrow 0$,   and $\chi\rightarrow 0$ with both phenomenology, e.g.,  \cite{F95}, \cite{P00}  and \cite{L07}, and the rate proven for the Navier-Stokes equations in \cite{DC92}, \cite{DF02},  \cite{M94} and \cite{W97}.  On the other hand, the upper bound being independent of the viscosity at high Reynolds number  is in accord with the Kolmogrov's conventional turbulence theory.

This estimate gives insight into $\chi$ by asking model's dissipation, $\frac{\chi \delta^2}{U L}$, be comparable to the pumping rate of energy to small scales by the nonlinearity, $2 \, \frac{U^3}{L}$ , and to the correction to the asymptotic, $\Rey \rightarrow \infty$, rate due to energy dissipation in the inertial range, $\Rey^{-1}\, \frac{U^3}{L}$. The comparison suggests the following range  for $\chi$,

$$\Rey^{-1} \, \frac{UL}{\delta^2} \leq \chi \leq 2 \, \frac{UL}{\delta^2}  \mbox{\hspace{50pt} mesh independent case.}$$

In large eddy simulation (LES)  the smallest scale available is $h$,  when the model is solved on a spacial mesh with mesh-width $h$.  On the other hand, Kolmogorov dissipation micro-scale, which determines the size of the smallest persistent solution scales, is $\Rey ^{-\frac{3}{4}} \, L $.  Hence, one can estimate $  h \simeq  \Rey ^{-\frac{3}{4}} \, L$. Morover, the scale $\delta$ is in general chosen to be of the order of the mesh size $h$ in a practical computation.     In  other words,  success for a turbulence simulation  minimally requires that $\delta = h =  \Rey ^{-\frac{3}{4}} \, L$. Therefore the following estimate of mesh dependence case can be derived,

$$  \frac{U}{L}\, (\frac{L}{h})^{\frac{2}{3}}\leq \chi \leq 2 \frac{U}{L}\, (\frac{L}{h})^{2}  \mbox{\hspace{46pt} mesh dependent case.}$$
Note that in both cases $\chi \rightarrow \infty$ as $\delta$ and $h \rightarrow 0$ which is consistent with results shown in \cite{LN07}.

\section{Preliminaries}\label{Section2}

This section is  devoted to standard  definitions and notations.  We restrict ourselves to what we need for our usage and we skip proofs and technical details. Throughout this article, the $L^2(\Omega)$ norm and inner product will be denoted by $\|\cdot\|$ and $(\cdot , \cdot)$. Likewise, the $L^{p}(\Omega)$ norms are denoted by $\|\cdot\|_{p}$.  $\grad u$ is the gradient tensor, $(\grad u)_{ij} =\frac{\partial u_j}{\partial x_i}$,  for $i, j= 1, 2 , 3$.
 \subsection{Differential Filter and LES}
 In any turbulent flow, it is expected that large  scales of motion contain the bulk of a flow's kinetic  energy, and account for most of the momentum transport \cite{P00}.   Large Eddy Simulations (LES) aim to  compute only large flow structures (larger than the filter width $\delta$). This can be accomplished by  removing the  small flow scales from the solution by  a spatial low-pass filtering. Accordingly,   the mean effects of these  small scales' random character  on the large eddies has to be modeled.  To introduce any LES model (the time relaxation model here),  a local spacial averaging operator associated with a length-scale $\delta$ must be selected, and many are possible. These are well documented in the literature,    e.g.,  \cite{BIL06} and \cite{J04}.  We chose a continuous differential filter, Germano \cite{G86},  as the follows.

Given an $L-$periodic $\phi(x) \in L^2 (\Omega)$ and a filtering redius of $\delta >0 $, its average $\overline{\phi}$ is the unique $L-$ periodic solution of the PDE,

\begin{equation} 
\begin{split}
- \delta^2 \triangle \overline{\phi} + & \overline{\phi} = \phi  \hspace{10pt}  \hspace{10pt} \mbox{in}\,\, \Omega, \\
& \overline{\phi} = \phi \hspace{20pt} \mbox{on}\,\, \partial \Omega.
\end{split}
\end{equation}
The filter size $\delta$  is in general chosen to be of the order of the mesh size $h$ in a practical computation \cite{BIL06}. This filtering operation is often denoted  $\overline{\phi} = G \phi$ where $G = (I - \delta^2 \triangle)^{-1}$. It is important in many applications to obtain the unfiltered solution from the filtered solution.  However, the filter  is non-regular because its inverse is unbounded \cite{SA99},  an approximate inverse can be obtained.   From here, the basic problem is: given $\overline{\phi}$   find \textit{useful} approximations of $\phi$. Stolz and Adams in \cite{SA99} proposed a method (ADM) based on a repeated application of the filter to approximately deconvolve the filtered
solution and they  applied this model successfully for the LES.

\subsection{Approximate de-convolution Model}   The de-convolution problem becomes, 
$$\mbox{Given} \hspace{4pt} \overline{\phi}, \hspace{10pt}  \mbox{solve} \hspace{4pt}  G\phi = \overline{\phi} \hspace{10pt} \mbox{for} \hspace{4pt} \phi.$$
It is central in both image processing  \cite{BB98} and turbulence modeling in large eddy simulation \cite{G97}.   $G$ is not invertible or at least not stably invertible due to small divisor problems. Thus, this de-convolution problem is ill posed. Hence finding an appropriate approximation becomes needful for the applications.

The \textit{van Cittert algorithm},  was first  used for image reconstruction by van Cittert in 1931,  is a well-known procedure in regularizing ill-posed problems. Consider a filter $G$ and a filtered function $\overline{\phi}$. Assign $\phi_0 = \overline{\phi}$, then for $n = 0, 1, 2, ..., N-1$ perform the following fixed-point iteration,

$$\phi_{n+1} = \phi_n + \{\overline{\phi} - G \phi_n\}.$$
This is the first order Richardson iteration for the operator equation
$G\phi = \overline{\phi}$ involving a possibly noninvertible operator $G$. For each $N = 0, 1, . . .$,   the algorithm computes an approximate solution $\phi_N$ to the above de-convolution equation by $N$
steps of a fixed-point iteration. Since the de-convolution problem is ill
posed, convergence as $N \rightarrow \infty$ is not expected.

\begin{definition} The Nth van Cittert approximate deconvolution operator $G_N: L^2 (\Omega) \rightarrow L^2 (\Omega)$ is defined as, 

$$G_N (\overline{\phi}) \coloneqq \phi_N.$$ 
We then can rewrite, 
$$G_N G \phi = \phi_N.$$
\end{definition}
The algorithm can be simplified to obtain an explicit formula for the  $N^{\mbox{th}}$ de-convolution operator $G_N$,

$$G_N \phi  = \sum_{n=0}^{N} (I - G)^n \phi.$$
The (bounded) operator $G_N$ is an approximation to the (unbounded)
inverse of the filter $G$ in the  following sense, $G_N \simeq G^{-1}$.

\begin{lem} \label{lem1}(Error in approximate de-convolution)
For any $\phi \in L^2 (\Omega),$

\begin{equation}
\begin{split}
 \phi - G_N \overline{\phi} &= \big( (-1)^{N+1} \Delta^{N+1} \delta^{2N+2}\big) G^{N+1} \phi\\ 
 &= \mathcal{O} (\delta ^{2N+2}) \hspace{50pt} \mbox{as} \hspace{5pt}  \delta \rightarrow 0.
 \end{split}
\end{equation}
\end{lem}
\begin{proof}
See \cite{BIL06}.
\end{proof}

\subsection{Time Relaxation Model} Time Relaxation Models (\ref{TRM}) were introduced by Stolz, Adams and Kleiser in \cite{SAK01-1} and \cite{SAK01-2}. The model's solutions $u(x,t)$ are intended to approximate the true flow averages. Accordingly,  the effect of nonrepresented scales is modeled by a relaxation regularization involving a repeated filter operation $G_N$ and a dynamically estimated relaxation parameter $\chi$.  Considering an incompressible flow in  a periodic box $\Omega = (0,\ell)^3$, the resulting models are given by:

\begin{equation} \label{TRM}
\begin{split}
u_t+ u \cdot \nabla u -\nu \Delta u + & \nabla p + \chi (u - G_N \overline{u}) = f(x)\hspace{10pt} \mbox{and}\hspace{10pt} \grad \cdot u =0 \hspace{10pt} \mbox{in}\,\, \Omega, \\
& u(x,0)=u_0(x) \hspace{10pt} \mbox{in}\,\, \Omega,
\end{split}
\end{equation}
periodic boundary conditions are imposed, 

\begin{equation} \label{BC}
u(x+\ell e_j,t)=u(x,t)\hspace{10pt}\mbox{for any }\hspace{10pt} j=1,2,3, 
\end{equation}
the data $u_0(x)$ and  $f(x)$ are smooth, $\ell$ -periodic and divergence free. We restrict attention to mean-zero body forces and initial conditions so the velocity remains mean-zero for all $t>0$,

\begin{equation}\label{DataConditions} 
\begin{split}
&\hspace{5pt}\grad \cdot f= 0 \hspace{12pt}\mbox{and  }\hspace{12pt} \grad \cdot u_0=0, \\
& \int_{\Omega} \kappa \,dx =0 \hspace{10pt} \mbox{for any }  \hspace{10pt} \kappa = u, u_0, f, p.
\end{split}
\end{equation}
Existence, uniqueness  and regularity of strong solutions are described in  \cite{BL12} and \cite{LN07}.  The term $u - G_N \overline{u} = (I - G_N G)u$ was devised to inject extra energy dissipation to the computed solution of the unregularized NSE, and derive fluctuations below $\mathcal{O}(\delta)$ to zero  exponentially fast as $t \rightarrow \infty$ without altering the dominant scales $> \mathcal{O}(\delta)$.  It is shown in \cite{LN07} that the fluctuations below $\mathcal{O}(\delta)$ must $\rightarrow 0$ in $L^2(\Omega \times (0,T))$ as $\chi \rightarrow \infty$.

The convergence of the Finite Element discretization of the TRM is presented in \cite{ELN07} along with a numerical study which shows that the time relaxation term does not alter shock speeds in the inviscid compressible case. A better performance of the TRM is reported in the study of the flow past a full step problem  in  \cite{DN14} and \cite{ELN07}.  The model also performed extremely well in a posteriori tests for incompressible channel flow \cite{SAK01-2}. Adams \textit{et al.}  have performed extensive computational tests of the time relaxation model on the compressible flows with shocks  in \cite{SAK01-1}, and on the   compressible decaying isotropic turbulence in \cite{SA99}. A significant improvement over established subgrid scale models are reported in all of these works.

\subsection{Energy Estimates}
To start a standard energy calculation, multiply (\ref{TRM}) by $u$, integrate over the domain $\Omega$  and then integrate with respect to time from $0$ to $t$. We have the following proposition on the existence and uniqueness of the weak  and strong solutions. 

\begin{prop}
Let $u_0 \in L_0^2 (\Omega), f \in L^2 (\Omega \times (0,T))$ and $\int_{\Omega} f \, dx = 0$. There exists a weak solution to (\ref{TRM}). The solution is unique if it is additionally a strong solution. Moreover, if $u$ is a strong solution it satisfies the energy equality:

\begin{equation} \label{EnergyEq}
\frac{1}{2} \|u(t)\|^2 + \int_0^t \int_{\Omega} \nu |\nabla u|^2 + \chi (u - G_N \overline{u}) \cdot u \, dx \, dt' =  \frac{1}{2} \|u_0\|^2 + \int_0^t \int_{\Omega} f \cdot u \, dx \, dt'.
\end{equation}
\end{prop}
\begin{proof}
See \cite{LN07}.
\end{proof}

\begin{re}
Weak solutions satisfies the energy inequality  which $''=''$ replaced by $''\leq''$ in (\ref{EnergyEq}). 
\end{re}

Since the operator $(I - G_N G)$ is Hermition and symmetric Positive Definite \cite{SAK01-2}, the relaxation term is purely dissipative.  Consider the operator $B: L^2 (\Omega) \rightarrow L^2 (\Omega)$ satisfying:
\begin{equation}\label{OperatorB}
B^2 \phi  \coloneqq \delta^{-(2N+2)} (I - G_N G) \phi = \delta^{-(2N+2)} (\phi - G_N \overline{\phi}).
\end{equation}
Hence $B = \delta ^{-(N+1)} \sqrt{(I - G_N G)}$ is well-defined, positive and bounded. Moreover we have,  

\begin{equation}\label{B-Operator}
(\phi - G_N \overline{\phi}, \phi) = \delta ^ {2N+2} (B \phi, B \phi) = \delta ^ {2N+2} \| B \phi \|^2.
\end{equation}

Because   $I - G_N G$ is a positive definite operator, considering the energy equality  (\ref{EnergyEq}), the model's relaxation term $\chi (u - G_N \overline{u})$ extracts energy from resolved scales and dissipate through time scales of motion.  Thus  the model energy dissipation rate (per unit volume)  includes dissipation due to the viscous forces and the model's diffusion which is given by, 

$$\varepsilon = \varepsilon_0 + \varepsilon_M,$$
where,
$$ \varepsilon_0  \coloneqq \frac{1}{|\Omega|} \int_{\Omega} \nu |\nabla u|^2   =  \frac{1}{|\Omega|}\,  \nu \|\nabla u\|^2 ,$$
and,
$$ \varepsilon_M \coloneqq \frac{1}{|\Omega|}  \int_{\Omega} \chi (u - G_N \overline{u}) \cdot u  = \frac{1}{|\Omega|}\, \chi \, \delta^{2N+2} \, \|Bu\|^2. $$
We will consider time-averaged quantity using the notation,
$$\langle\psi(\cdot)\rangle \coloneqq  \limsup\limits_{T\rightarrow\infty}  \, \frac{1}{T} \int_{0}^{T} \psi(t)\, dt.$$
Thus the time-averaged energy dissipation rate for (\ref{TRM}) is,

$$\langle\varepsilon \rangle=  \limsup\limits_{T\rightarrow\infty}  \frac{1}{|\Omega|}\, \frac{1}{T} \int_{0}^{T}   \nu\|\nabla u\|^2  + \chi \, \delta^{2N+2} \, \|Bu\|^2\, dt.$$

\begin{re}\label{Remark1}
Using Poincare’s inequality, together with the Cauchy–Schwarz and Gr{\"o}nwall’s inequalities in (\ref{EnergyEq}) imply that the kinetic energy is uniformly bounded in time,
$$\sup_{t \in (0, \infty)} \|u (t)\|^2 \leq C\, (\mbox{data}) < \infty,$$
and it follows that, 
$$   \frac{1}{T} \int_0^T \varepsilon (u) \, dt  \leq C \, (\mbox{data})  < \infty,$$
which means $\langle \varepsilon (u)\rangle$ is well-defined. 
\end{re}
\subsection{Dimensionless Numbers} To study the time relaxation model precisely, it is critical to find the model's equivalent of the large scales' Reynolds number of the Navier-Stokes equations. The Reynolds number for the Navier-Stokes equations is the ratio of non-linearity (inertia) to viscous (friction) terms action on the largest scales, 
$$\Rey \simeq \frac{|u \cdot \nabla u|}{|\nu \Delta u|} \simeq \frac{ U \frac{U}{L}}{\nu \frac{U}{L^2}} = \frac{UL}{\nu}.$$
 The ratio of non-linearity to dissipative effects should be the analogous quantity. Since the time relaxation term acts to dissipate energy, the new quantity should correspond to, 
$$R_N \simeq \frac{|u \cdot \nabla u|}{ |\chi (u - G_N \overline{u})|}.$$
Proceeding analogously, Layton and Neda \cite{LN07} proposed the following dimensionless parameter for the model.  This derivation is under the assumption that viscous dissipation is negligible compared to dissipation due to time relaxation. Using Lemma \ref{lem1} and the fact that for large scales $(\frac{\delta}{L})^2 \ll 1$, we have,  

{\begin{equation}
\begin{split}
R_N \simeq \frac{|u \cdot \nabla u|}{ |\chi (u - G_N \overline{u})|} & = \frac{|u \cdot \nabla u|}{|\chi ( I -  [- \delta ^2 \Delta + I]^{-1})^{N+1} u |}\\
& =  \frac{|u \cdot \nabla u|}{|\chi \delta^{2N+2} \Delta^{N+1} [- \delta ^2 \Delta + I]^{-(N+1)} u  |}\\
& \simeq \frac{U \frac{U}{L}}{ \chi \delta^{2N+2} (\frac{1}{L^2})^{N+1} [\frac{\delta^2}{L^2} +1]^{-(N+1)} U}\\
& \simeq \frac{ U L^{2N+1}}{\chi \, \delta^{2N+2}}.
\end{split}
\end{equation}

\begin{definition}\label{Def-RN}
The dimensionless time relaxation parameter $R_N$ for the time relaxation model (\ref{TRM}) is,
$$R_N \coloneqq \frac{ U L^{2N+1}}{\chi \, \delta^{2N+2}}.$$
\end{definition}
\section{Energy dissipation rate estimates}\label{Section3}

With $|\Omega|$ the volume of the flow domain, the scale of the body force, large scale velocity,  and length, $F, U, L$, are defined as,

\begin{equation} \label{Scales}
\begin{split}
&F\coloneqq \langle\frac{1}{|\Omega|} \|f\|^2\rangle^{\frac{1}{2}},\\
&U\coloneqq \langle\frac{1}{|\Omega|} \|u\|^2\rangle^{\frac{1}{2}},\\
&L\coloneqq \min \big\{|\Omega|^{\frac{1}{3}} , \, \frac{F}{\langle\frac{1}{|\Omega|} \|\grad f\|^2\rangle^{\frac{1}{2}}},\, \frac{F}{\|\grad f\|_{L^{\infty}(0,T;L^{\infty}(\Omega))}},\, \frac{F^{\frac{1}{N+1}}}{\langle \frac{1}{|\Omega|} \|Bf\|^2 \rangle^{\frac{1}{2N+2}}}\big\}.
\end{split}
\end{equation}
One can show that $U, F$ and $L$ have units of [length $\times$ time$^{-1}$],  [mass $\times$ length  $\times$ time$^{-2}$] and [length] respectively for fixed $N$. For example, recalling (\ref{OperatorB}), we have 

$$|Bf|^2 =|\delta^{-2N-2} (I - G_N G) f^2| = |\delta^{-2N-2} \, \delta^{2N+2} \frac{1}{L^{2N+2}} (\frac{\delta^2}{L^2} +1)^{-(N+1)} F^2 | ,$$
since $ \frac{\delta^2}{L^2}\ll 1$ for the large scales, then $(\frac{\delta^2}{L^2} +1)$ is $\mathcal{O}(1)$. From here, it is easy to see that  $\langle \frac{1}{|\Omega|} \|Bf\|^2 \rangle^{\frac{1}{2N+2}}$ has the same units as $F^{\frac{1}{N+1}} \, L^{-1}$,  which shows that the fourth element of the length scale has units of length. Therefore  $L$ has units of length and satisfies, 

\begin{equation} \label{Scales2}
\begin{split}
&{\langle\frac{1}{|\Omega|} \|\grad f\|^2\rangle^{\frac{1}{2}}} \leq \frac{F}{L},\\
&{\|\grad f\|_{L^{\infty}(0,T;L^{\infty}(\Omega))}} \leq \frac{F}{L},\\
&{\langle \frac{1}{|\Omega|} \|Bf\|^2 \rangle^{\frac{1}{2}}} \leq \frac{F}{L^{N+1}}.
\end{split}
\end{equation}

\begin{thm}
Let $u(x,t)$ be a mean-zero solution of the Time Relaxation Model (\ref{TRM}) with the periodic boundary conditions (\ref{BC}) and the data conditions (\ref{DataConditions}). Then the time averaged energy dissipation rate per unit mass satisfies,

$$\langle \varepsilon (u) \rangle \leq \bigg( 2 + \frac{1}{\Rey}  + \frac{1}{R_N} \bigg)\, \frac{U^3}{L},$$
where $U$ and $L$ are defined in (\ref{Scales}) and $R_N$ is defined by, 
$$R_N =  \frac{ U L^{2N+1}}{\chi \, \delta^{2N+2}}.$$
\end{thm}

\begin{proof}

The proof is a synthesis of the model's energy balance (\ref{EnergyEq}) , the breakthrough arguments of Doering and Foias \cite{DF02} from the NSE case with careful treatment of the time relaxation term.  Considering $\|u(t)\|^2$ being bounded in time, average (\ref{EnergyEq}) over $[0, T ]$, applying the Cauchy-Schwarz inequality in time yields,

\begin{equation}
\begin{split}
\frac{1}{T} \int_{0}^{T} \varepsilon dt  & = \mathcal{O}(\frac{1}{T}) +   \frac{1}{T}  \frac{1}{|\Omega|}\int_{0}^{T} (f,u(t))  dt\\
 &\leq  \mathcal{O}(\frac{1}{T}) +  ( \frac{1}{|\Omega|}\frac{1}{T} \int_{0}^{T} \|f\|^2 dt)^{\frac{1}{2}} \, ( \frac{1}{|\Omega|}\frac{1}{T} \int_{0}^{T} \|u\|^2 dt)^{\frac{1}{2}}.
\end{split}
\end{equation}
Taking the limit superior, which exists (Remark \ref{Remark1}),  as $T\rightarrow \infty$  we obtain,

\begin{equation}\label{FirstBound}
\langle\varepsilon \rangle \leq U\, F.
\end{equation}

Next, multiply the time relaxation model (\ref{TRM}) by $f$, integrate
over $\Omega$ and integrate by parts as appropriate. Then take the time average to obtain,

\begin{equation}\label{BoundF}
\langle \frac{1}{|\Omega|} \|f\|^2 \rangle = \frac{1}{|\Omega|} \, \langle \, (u_t, f) + \nu (\nabla u, \nabla f) + (u \otimes u , \nabla f) + (\chi (u - G_N \overline{u}), f)\,  \rangle.
\end{equation}

First note that the time average of the time derivative vanishes as $T\rightarrow \infty$. The second and third terms on the right hand side are bounded using the Cauchy-Schwarz-Young inequality and (\ref{Scales2}) by,

\begin{equation}\label{BoundF1}
| \frac{1}{T} \int_0^T \frac{1}{|\Omega|} (u \otimes u, \nabla f)\, dt| \leq \|\nabla f\|_{L^{\infty}} \frac{1}{T} \int_0^T \frac{1}{|\Omega|} \|u\|^2\, dt \leq \frac{F}{L}\, U^2 \mbox{\hspace{30pt}as\hspace{30pt}} T\rightarrow \infty.
\end{equation}
And, 

\begin{equation}\label{BoundF2}
\begin{split}
|\frac{1}{T} \int_0^T \frac{\nu}{|\Omega|} (\nabla u. \nabla f) \,dt | &\leq \big( \frac{1}{|\Omega|}\frac{1}{T} \int_0^T \nu^2 \|\nabla u\|^2\, dt\big)^\frac{1}{2}\, \big( \frac{1}{|\Omega|}\frac{1}{T} \int_0^T \|\nabla f\|^2 \, dt\big)^\frac{1}{2}\\
& \leq \langle \varepsilon_0 \rangle ^{\frac{1}{2}} \nu^{\frac{1}{2}}\, \frac{F}{L} \mbox{\hspace{30pt}as\hspace{30pt}} T\rightarrow \infty\\
& = F\, \big [  \frac{\langle \varepsilon_0 \rangle ^{\frac{1}{2}}}{U^{\frac{1}{2}}} \, \frac{U^{\frac{1}{2}} \, \nu^{\frac{1}{2}}}{L} \big]\\
&\leq  F \, \big[  \frac{1}{2} \frac{\langle \varepsilon_0 \rangle}{U}  + \frac{1}{2}  \frac{U\, \nu}{L^2}\big].
\end{split}
\end{equation}

Next, considering $B$ being a self-adjoint operator,  we use  the Cauchy-Schwarz and Young's inequality with  (\ref{Scales2}) to see,

\begin{equation}
\begin{split}
\frac{1}{|\Omega|} \, \frac{1}{T} \int_0^T \,&(\chi (u - G_N \overline{u}), f) \,dt  \leq  \frac{1}{|\Omega|} \, \frac{1}{T} \int_0^T \, \chi  \delta^{2N+2}\, \| B u\| \|Bf\|\, dt\\
&\leq  \big(  \frac{1}{|\Omega|} \, \frac{1}{T} \int_0^T \chi \delta^{2N+2} \|B u \|^2 \, dt\big)^{\frac{1}{2}} \,  \big(  \frac{1}{|\Omega|} \, \frac{1}{T} \int_0^T \|B f \|^2 \, dt\big)^{\frac{1}{2}} \,  \chi^{\frac{1}{2}} \delta^{N+1}.
\end{split}
\end{equation}

Inserting multipliers of $\frac{1}{\sqrt{U}}$ and $\sqrt{U}$  in the two terms and taking the limit superior as $T \rightarrow \infty$, we have,

\begin{equation} \label{BoundF3}
\begin{split}
\langle  \, (\chi (u - G_N \overline{u}), f)\,  \rangle & \leq \frac{\langle \varepsilon_M  \rangle^{\frac{1}{2}}}{\sqrt{U}} \, \frac{F }{L^{N+1}}  \chi^{\frac{1}{2}} \delta^{N+1} \sqrt{U}\\
& \leq F \big[ \frac{1}{2} \frac{\langle \varepsilon_M  \rangle}{U} + \frac{1}{2} U  \, \chi \, \frac{\delta^{2N+2}}{L^{2N+2}}\big]. 
\end{split}
\end{equation}

Combining the identity from (\ref{BoundF})  with the estimates in (\ref{BoundF1}), (\ref{BoundF2}) and (\ref{BoundF3}),
\begin{equation}\label{SecondBound}
F \leq \frac{U^2}{L} +  \frac{1}{2} \frac{\langle \varepsilon \rangle}{U} + \frac{1}{2} \frac{U \nu}{L^2} + \frac{1}{2}  \chi  U \frac{\delta^{2N+2}}{L^{2N+2}}. 
\end{equation}

Finally,  using the above estimate on (\ref{FirstBound}), we obtain, 

$$\langle \varepsilon  \rangle \leq \frac{U^3}{L} + \frac{1}{2} \langle \varepsilon  \rangle + \frac{1}{2} \frac{U^2 \nu}{L^2} + \frac{1}{2}  \chi  U^2 \frac{\delta^{2N+2}}{L^{2N+2}}. $$

Thus, recalling Definition \ref{Def-RN},   as claimed,
$$\langle \varepsilon  \rangle \leq 2\, \frac{U^3}{L} + \frac{1}{\Rey} \frac{U^3}{L} + \frac{1}{R_N} \frac{U^3}{L}.$$

\end{proof}

\begin{cor}
Tracking back the analysis, it can be seen that the nonlinearity pumps energy to small scales, $2 \, \frac{U^3}{L}$, and the rate of the dissipation due to the viscosity is $\Rey^{-1} \, \frac{U^3}{L}$. From here,  $\chi$ can be estimated by comparing the model's rate of dissipation ${R_N}^{-1} \frac{U^3}{L} $ to the above rates. The comparison yields, 
\begin{equation}
\Rey^{-1} \, \frac{U}{L}\,  (\frac{L}{\delta})^{2N+2} \leq \chi \leq  2 \, \frac{U}{L}\,  (\frac{L}{\delta})^{2N+2}, 
\end{equation}
for fixed $N$. In Large Eddy Simulation, the filter size $\delta$ is mostly chosen to be of the order of the mesh size $h$ which is the smallest available scale. On the other hand, smallest scales in turbulent flow can be computed by the Kolmogorov microscale $\Rey^{- \frac{3}{4}} \, L$.  Thus,  success for a turbulence simulation  minimally requires that $\delta \simeq h \simeq  \Rey^{-\frac{3}{4}} \, L$, and this leads to the following estimation on $\chi$ for mesh dependence case, 

\begin{equation}\label{Rangediscrete}
\frac{U}{L} \, (\frac{L}{h})^{2N+ \frac{10}{3}} \leq \chi \leq  2\, \frac{U}{L} \, (\frac{L}{h})^{2N+2}.
\end{equation}
\end{cor}

\section{Conclusion}\label{Section4}

In this paper,  we delineated time averaged energy dissipation rate for the Time Relaxation Model for $3d$ turbulence in a box driven by a persistent body force with periodic boundary conditions.  Motivated by the analysis,   a narrowing of the commonly accepted ranges of the parameter is proposed.  The analysis does not apply to $2d$ flows, laminar flows, turbulence generated by shear flows,  and decaying turbulence. These cases are interesting open problems.

Turbulence models, like the Time Relaxation model here, are introduced to account for sub-mesh scale effects, when solving  fluid flow problems numerically on an under-resolved spatial mesh size $h$. Therefore, it is necessary to calculate  the energy dissipation in the turbulence model discretized on a coarse mesh, see e.g. \cite{AP18}. Answering this question could lead to a narrowing of the range of the parameter $\chi$ in (\ref{Rangediscrete}).  In typical discretizations  the conservation of mass is only weakly enforced, leading to discrete solutions $u^h$ which have $\nabla \cdot u^h \neq 0$. This leads to a second nonlinear term $- \frac{1}{2} (\nabla \cdot u^h) u^h$. The parameter $\chi$ then might affect the rate at which $- \frac{1}{2} (\nabla \cdot u^h) u^h$  pumps energy to smaller scales.

\bigskip

\noindent{\bf Acknowledgments.} 
A.P. was Partially supported by NSF grants, DMS 1522267 and CBET 1609120.


\begin{thebibliography}{9}



  \bib{BB98}{article}{
   author={M. Bertero},
   author={B. Boccacci},
   title={Introduction to Inverse Problems in Imaging},
   journal={IOP Publishing Ltd},
   date={1998},
   
   }


\bib{BIL06}{book}{
   author={L.C. Berselli},
   author={T. Iliescu},
   author={W.J. Layton},
   title={Mathematics of large eddy simulation of turbulent flows},
   series={Scientific Computation},
   publisher={Springer-Verlag, Berlin},
   date={2006},
   pages={xviii+348},
 
}

 
\bib{BL12}{article}{
   author={L.C. Berselli},
   author={R. Lewandowski},
  title={Convergence of approximate deconvolution models to the mean Navier–Stokes equations},
 journal={Annales de l'Institut Henri Poincaré C, Analyse non linéaire},
   volume={29},
   issue={2},
   date={2012},
   pages={171--198},
 
}

\bib{CL10}{article}{
   author={Connors, J},
   author={Layton, W.},
  title={ On the Accuracy of the Finite Element Method Plus Time Relaxation},
 journal={Math. Comput.},
   volume={79},
 
   date={2010},
   pages={619–648},
 
}




  
\bib{DLPRSZ18}{article}{
   author={V. DeCaria},
   author={W. Layton},
    author={A. Pakzad},
     author={Y. Rong},
      author={N. Sahin},
       author={H. Zhao},
   title={On the determination of the grad-div criterion},
   journal={Journal of Mathematical Analysis and Applications},
   volume={467},
   issue={2},
   date={2018},
   pages={1032--1037},
   }
   
 
   
   
\bib{DC92}{article}{
   author={C.R. Doering},
   author={P. Constantin},
   title={Energy dissipation in shear driven turbulence},
  
   journal={Physical review letters},
   volume={69},
   date={1992},
   pages={1648},
   
}
   
   
   
\bib{DG95}{book}{

   author={C.R. Doering},   
      author={J.D. Gibbon},
       
    title={Applied Analysis of Navier–Stokes Equations},
      publisher={Cambridge University Press},
  year={1995},
   
   place={Cambridge},
  
}

   
     \bib{DF02}{article}{
   author={C.R.  Doering},
   author={C. Foias},
   title={Energy dissipation in body-forced turbulence},
   journal={J. Fluid Mech.},
   volume={467},
   issue={2},
   date={2002},
   pages={ 289-306},
   }


   \bib{DN14}{article}{
   author={Dunca, A.},
   author={Neda, M},
   title={Numerical analysis of a nonlinear time relaxation model of fluids},
   journal={J. Math. Anal. Appl.},
   volume={420},
   date={2014},
   pages={1095–1115},
   }
   
   
   \bib{ELN07}{article}{
  author={V. Ervin},
  author={W. Layton},
   author={M. Neda},
   title={Numerical analysis of a higher order time relaxation model of fluids},
  journal={nternational Journal of Numerical Analysis and Modeling},
   volume = {4},
pages = {648-670},
year = {2007},  
  
}


\bib{FMRT01}{book}{

   author={Foias, C.},   
      author={Manley, O.},
         author={Rosa, R.},
            author={Temam, R.},
  title={Navier-Stokes Equations and Turbulence},
   publisher={Cambridge University Press},
  year={2001},
   series={Encyclopedia of Mathematics and its Applications},
   place={Cambridge},
   DOI={10.1017/CBO9780511546754},
}

\bib{F95}{book}{

   author={Frisch, U.},   
     title={Turbulence: the Legacy of A. N. Kolmogorov},
       publisher={Cambridge University Press},
   year={1995},
  place={Cambridge},
}




1995 . Cambridge University Press.
   
   \bib{G86}{article}{
   author={M. Germano},
  
   title={Differential filters of elliptic type},
   journal={Phys. Fluids},
   volume={29},
   issue={6},
   date={1986},
   pages={1757–1758},
   }
   
  
   
   \bib{G97}{article}{
   author={B.J.  Geurts},
   title={Inverse modeling for large eddy simulation},
   journal={Phys. Fluids},
   volume={9},
   
   date={1997},
   pages={3585},
   }



\bib{H72}{article}{
   author={L.N. Howard},
   title={Bounds on flow quantities},
  
   journal={Annual Review of Fluid Mechanics},
   volume={4},
   date={1972},
 
   pages={473--494},
   
}



\bib{J04}{book}{
   author={V. John},
 title={Large Eddy Simulation of Turbulent Incompressible Flows},

   publisher={Springer, Berlin},
   date={2004},

 
}

\bib{K97}{article}{
   author={R.R. Kerswell},
   title={Variational bounds on shear-driven turbulence and
turbulent Boussinesq convection },
   journal={Physica D},
   volume={100},
   date={1997},
   pages={355--376},
  
}

  
\bib{AP16}{article}{
   author={A. Pakzad},
   title={Damping functions correct over-dissipation of the Smagorinsky Model},
  journal={ Mathematical Methods in the Applied Sciences},
   
volume = {40},
number = {16},
issn = {1099-1476},
pages = {5933--5945},
year = {2017},  
  
}

  
\bib{AP18}{article}{
   author={Pakzad, A.},
   title={Analysis of mesh effects on turbulence statistics},
  journal={Journal of Mathematical Analysis and Applications},
  volume = {475},
  pages = {839-860},
    year = {2019},  
  
}





\bib{P00}{book}{
   author={S. Pope},
   title={Turbulent Flows},
  
   publisher={Cambridge Univ. Press},
   place={Cambridge, UK},
   date={2000},
} 




\bib{L07}{article}{
   author={W. Layton },
   title={ Bounds on energy and helicity dissipation rates of approximate
deconvolution models of turbulence},
   journal={SIAM J Math. Anal.,},
   volume={39},
   date={2007},
 
   pages={916-931},
 
}

\bib{L02}{article}{
   author={Layton, W.},
   title={Energy dissipation bounds for shear flows for a model in large
   eddy simulation},
   journal={Math. Comput. Modelling},
   volume={35},
   date={2002},
   number={13},
   pages={1445--1451},
 
}


\bib{L16}{article}{
   author={W. Layton},
   title={Energy dissipation in the Smagorinsky model of turbulence},
   journal={Appl. Math. Lett.},
   volume={59},
   date={2016},
   pages={56--59},

}
 


\bib{L08}{book}{
   author={Layton, W.},
   title={Introduction to the numerical analysis of incompressible viscous
   flows},
   series={Computational Science \& Engineering},
   volume={6},
   note={With a foreword by Max Gunzburger},
   publisher={Society for Industrial and Applied Mathematics (SIAM),
   Philadelphia, PA},
   date={2008},
   pages={xx+213},
   
}

  
\bib{LN07}{article}{
  author={W. Layton},
   author={M. Neda},
   title={Truncation of scales by time relaxation},
  journal={J. Math. Anal. Appl.},
   
volume = {325},
issue={2},
pages = {788-807},
year = {2007},  
  
}


  
\bib{LRS10}{article}{
  author={W. Layton},
   author={L. Rebholz},
   author={M. Sussman},
   title={Energy and helicity
dissipation rates of the NS-alpha and NS-omega deconvolution models},
  journal={IMA Journal of Applied Math},
   
volume = {75},

pages = {56-74},
year = {2010},  
  
}



  


 \bib{MW06}{article}{
   author={Majda, A. J.},
      author={Wang, X.},
   title={Nonlinear Dynamics and Statistical Theory for Basic Geophysical
Flows},
   journal={Cambridge University Press},
   place={ Cambridge, England},
  date={2006},
  
}

 \bib{M94}{article}{
   author={C. Marchioro},
   title={Remark on the energy dissipation in shear driven turbulence},
   journal={Phys. D},
   volume={74},
   date={1994},
   number={3-4},
   pages={395--398},
 
}
\bib{NPW15}{article}{
  author={Neda, M.},
  author={Pahlevani, F.},
    author={Waters, J.},
   title={Sensitivity Analysis of the Time Relaxation Model},
  journal={Appl. Math. Mech.},
   volume = {7},
   
pages = {89–115},
year = {2015},  
  }









\bib{SA99}{article}{
  author={S. Stolz},
  author={N.A. Adams},
   title={An approximate deconvolution procedure for large eddy simulation},
  journal={Phys. Fluids},
   volume = {11},
   issue = {7},
pages = {1699–1701},
year = {1999},  
  }


\bib{SAK01-1}{article}{
  author={S. Stolz},
  author={N.A. Adams},
   author={L. Kleiser},
   title={The approximate deconvolution model for LES of compressible flows and its application to shock-turbulent-boundary-layer interaction},
  journal={Phys. Fluids},
   volume = {13},
   number = {10},
pages = {2985-3001},
year = {2001},  
  
}

  
\bib{SAK01-2}{article}{
  author={Stolz, S.},
  author={N.A. Adams},
   author={L. Kleiser},
   title={An approximate deconvolution model for large eddy simulation with application to wall-bounded flows},
  journal={Phys. Fluids},
   volume = {13},
   number = {4},
pages = {997-1015},
year = {2001},  
  
}








\bib{W10}{article}{
   author={X. Wang},
   title={Approximation of stationary statistical properties of dissipative dynamical systems: time discretization},
   journal={Mathematics of Computation},
   volume={79},
   date={2010},
   number={269},
   pages={259-280},
  
}

\bib{W00}{article}{
   author={Wang, X.},
   title={Effect of tangential derivative in the boundary layer on time averaged energy dissipation rate},
   journal={Physica D: Nonlinear Phenomena},
   volume={144},
   date={2000},
   issue={1-2},
   pages={142--153},
  
}



\bib{W97}{article}{
   author={X. Wang},
   title={Time-averaged energy dissipation rate for shear driven flows in
   $\bold R^n$},
   journal={Phys. D},
   volume={99},
   date={1997},
   number={4},
   pages={555--563},
  
}
  
  
  
  

\end{thebibliography}
\end{document}